\documentclass[11pt,reqno]{amsart}
\usepackage{graphicx} 
\usepackage{hyperref}
\usepackage{fourier}
\usepackage{fullpage}
\usepackage{amsmath,amssymb,amsthm}
\usepackage[shortlabels]{enumitem}
\usepackage[mathscr]{euscript}
\usepackage{dutchcal}
\usepackage{upgreek}
\usepackage{comment}
\usepackage{csquotes}
\usepackage{tikz}
\allowdisplaybreaks

\makeatletter
\def\resetMathstrut@{%
  \setbox\z@\hbox{%
    \mathchardef\@tempa\mathcode`\(\relax
    \def\@tempb##1"##2##3{\the\textfont"##3\char"}%
    \expandafter\@tempb\meaning\@tempa \relax
  }%
  \ht\Mathstrutbox@1.2\ht\z@ \dp\Mathstrutbox@1.2\dp\z@
}
\makeatother

\newtheorem{theorem}{Theorem}
\newtheorem{lemma}[theorem]{Lemma}
\newtheorem{prop}[theorem]{Proposition}

\newtheorem{corollary}[theorem]{Corollary}

\newtheorem{conjecture}[theorem]{Conjecture}

\newtheorem{remark}[theorem]{Remark}

\newcommand{\R}{\mathbb{R}}
\newcommand{\Z}{\mathbb{Z}}
\newcommand{\N}{\mathbb{N}}

\newcommand{\norm}[1]{\| #1 \|}
\newcommand{\lesim}{\lesssim}

\renewcommand{\le}{\leqslant}
\renewcommand{\ge}{\geqslant}

\renewcommand{\setminus}{\smallsetminus}
\renewcommand{\subset}{\subseteq}
\newcommand{\cc}{\mathsf{c}}

\newcommand{\eqdef}{\stackrel{\mathrm{def}}{=}}

\newcommand{\vol}{\mathrm{vol}}

\renewcommand{\i}{\mathsf{i}}

\newcommand{\abs}[1]{\left\lvert #1 \right\rvert}

\newcommand{\X}{\mathbf X}
\newcommand{\Y}{\mathbf Y}
\newcommand{\Proj}{\mathsf{Proj}}

\newcommand{\NN}{\mathcal{N}}
\newcommand{\MM}{\mathcal{M}}

\newcommand{\sub}{\mathscr{C}}

\newcommand{\e}{\varepsilon}
\newcommand{\f}{\varphi}
\renewcommand{\d}{\delta}
\newcommand{\ud}[0]{\,\mathrm{d}}

\newcommand{\1}{\mathbf{1}}

\renewcommand{\nu}{\upnu}

\newcommand{\fS}{\mathfrak{S}}

\title{A threshold phenomenon for embeddings of Euclidean snowflakes and  impossibility of dimension reduction}

\date{}

\author{Assaf Naor}
\address{Department of Mathematics, Princeton NJ 08544-1000}
\email{naor@math.princeton.edu}

\author{Kevin Ren}
\address{Department of Mathematics, Princeton NJ 08544-1000}
\email{kr5621@princeton.edu}

\thanks{A.~N.~was supported  by NSF grant DMS-2453936 and a Simons Investigator award. K.~R.~was supported by an NSF GRFP fellowship, a Simons
Foundation Dissertation Fellowship in Mathematics, and a Cubist/Point72 PhD Fellowship.  }

\usepackage{xcolor}
\definecolor{maroon}{HTML}{AF3235}

\begin{document}

\begin{abstract} Fix $0<\theta\le 1$. We prove that if $1\le p \le 2/\theta$, then the $\theta$-snowflake of $\ell_2^k$, namely, $\R^k$ equipped with the metric $((x,y)\in \R^k\times \R^k)\mapsto \|x-y\|_2^\theta$,  embeds with distortion $O(1)$ into $\ell_p^m$ for some integer $m\lesssim_{p,\theta}k$, which is optimal as $k\to \infty$, as seen by comparing dimensions. However, for $p$ larger than the sharp threshold $2/\theta$ the following change in behavior occurs: If a $(1/\sqrt{k})$-dense subset of the Euclidean sphere     $S^{k-1}$ embeds into $\ell_p^m$ with  distortion $O(1)$, then necessarily $m\gtrsim_{p,\theta}( k/\log k)^{p\theta/2}$, which grows super-linearly in $k$ as $p\theta/2>1$, and this dimension bound is optimal as $k\to \infty$ up to lower order factors. We deduce from this statement that if $2<p<\infty$, then there exist arbitrarily large $n$-point subsets of $\ell_p$ with the property that if they embed with distortion $O(1)$ into $\ell_p^m$, then necessarily $m\gtrsim_p ((\log n)/(\log\log n)^2)^{p/2}$, thus demonstrating that the statement of the Johnson--Lindenstrauss dimension reduction lemma  fails to hold for $\ell_p$ .
\end{abstract}

\maketitle

\vspace{-0.3in}




\section{Introduction}

Given $D\ge 1$, a metric space $(\MM,d_\MM)$ is said to embed with distortion $D$ into a metric space $(\NN,d_\NN)$ if there are $f:\MM\to \NN$ and $s>0$ such that $s d_\MM(x,y)\le d_\NN(f(x),f(y))\le Ds d_\MM(x,y)$ for every $x,y\in \MM$. For $0< \theta\le 1$, the metric space $(\MM,d_\MM^\theta)$ is  called the  $\theta$-snowflake of $(\MM,d_\MM)$. If the Hausdorff dimension of $(\MM,d_\MM)$ equals  $\alpha$, then (by its definition) the Hausdorff dimension of its $\theta$-snowflake equals $\alpha/\theta$.\footnote{The (very rudimentary) properties of Hausdorff dimension that are mentioned herein can be found in e.g.~\cite{Rog98}. } Since bi-Lipschitz equivalent metric spaces have the same Hausdorff dimension, if the $\theta$-snowflake of $(\MM,d_\MM)$ embeds into $(\NN,d_\NN)$ with distortion $D$, then the Hausdorff dimension of $(\NN,d_\NN)$ must be at least $\alpha/\theta$.

The above standard comments show  that for every $k,m\in \N$, if the $\theta$-snowflake of $\R^k$ (equipped with any norm) embeds (with any distortion) into some $m$-dimensional normed space $\X$, then necessarily $m\ge k/\theta$. If $1\le p\le 2/\theta$, then the second part of  Theorem~\ref{thm:transition} below demonstrates that this is sharp (up to $p$-dependent constant factors) even when we equip $\R^k$ with the Euclidean metric $\|\cdot\|_2$ and  $\X=\ell_p^m$. However, the first part of Theorem~\ref{thm:transition} demonstrates that $2/\theta$ is the sharp threshold here, namely, for every fixed $2/\theta<p<\infty$  the target dimension $m$ must be at least a quantity that grows as $k\to\infty$  asymptotically faster  than the restriction that the aforementioned Hausdorff dimensional considerations  impose.

\begin{theorem}\label{thm:transition} 
Fix $p\ge 1$, an integer $k\ge 4$, and $0<\theta\le 1$. Suppose that $\MM$ is a $(1/\sqrt{k})$-dense subset of $S^{k-1}$.\footnote{The  notation and terminology that is used herein is standard. For example, given $\e>0$, a subset $\sub$ of a metric space $(\MM,d_\MM)$ is said to be $\e$-dense in $\MM$ if for every $x\in \MM$ there is $y\in \sub$ such that $d_\MM(x,y)\le \e$. Also,  $\|x\|_p=(|x_1|^p+\ldots+|x_n|^p)^{1/p}$ for every $x=(x_1,\ldots,x_n)\in \R^n$, and $\ell_p^n$ denotes the normed space $(\R^n,\|\cdot\|_p)$. The unit Euclidean sphere in $\R^n$ is $S^{n-1}=\{x\in \R^n:\ \|x\|_2=1\}$ and the normalized surface area probability measure on it will be denoted $\sigma_{n-1}$. The following standard asymptotic notation will also be used throughout:  $A\lesssim B$ stands for $A\le c B$ with $c$ a universal constant.  Correspondingly, we write $\lesssim_p$ when $c$ may depend only on $p$, and $A\asymp B$ for
$A\lesssim B\lesssim A$. We also use $O(B)$, $O_p (B)$, $\Omega(B)$ as a shorthand for $\lesim B$, $\lesim_p B$, $\gtrsim B$, respectively. } For every $D\ge1$ and $m\in \N$, if $(\MM,\|\cdot-\cdot\|_2^\theta)$ embeds with distortion $D$ into $\ell_p^m$, then the following lower bound on the target dimension $m$ must hold:
\begin{equation}\label{eq:snow-lower}
    m\ge\frac{1}{(4D)^p}\bigg(\frac{k}{2p\theta\log k+8}\bigg)^{\frac{p\theta}{2}}.
\end{equation}
  Conversely, $(\R^k,\|\cdot-\cdot \|_2^\theta)$ embeds with distortion $O(1)$ into $\ell_p^m$ for some integer $1\le m\lesssim_p (k/\theta)^{\max\left\{1,p\theta/2\right\}}$.
\end{theorem}

\subsection{Impossibility of dimension reduction} We will next describe consequences of Theorem~\ref{thm:transition}  to   dimension reduction (specifically, impossibility thereof). The Johnson--Lindentrauss (JL) lemma~\cite{JL} shows that any $n$-point subset of $\ell_2$ embeds with $O(1)$ distortion into $\ell_2^{m}$, where $m\lesssim \log n$. The JL lemma  is a result of great importance to, and impact on, multiple pure and applied areas; it would be futile to attempt to fully cover  herein the extensive work that has been done on this topic over the past four decades, and it will also be needlessly repetitive, as it is a very famous and well surveyed area; see e.g.~\cite{Vem04,Naor-ribe,BCFG23,DGLX25}.

An obvious mystery that arises from the JL lemma is whether it (or useful variants of it) hold in other spaces of interest. This question was posed at the inception of the JL lemma~\cite[Problem~3]{JL}, and has since been investigated extensively and repeated in multiple venues (see e.g.~\cite[Question~13]{Naor-ribe} and the discussion immediately following it). In particular, for all $p\in (1,\infty)\setminus \{2\}$ it was unknown if every $n$-point subset of $\ell_p$ embeds with distortion $O(1)$ into $\ell_p^m$ for $m\lesssim \log n$. The case $\theta=1$ of~\eqref{eq:snow-lower} demonstrates that this is not the case when $2<p<\infty$. Indeed, by a quick packing argument (see e.g.~\cite[Lemma~2.6]{MS86}) for every $k\in \N$ there is a $(1/\sqrt{k})$-dense  subset $\MM$ of $S^{k-1}$ such that if we set $n=|\MM|$, then  $\log n\asymp k\log k$. As $\ell_2$ is isometric to a subset of $L_p$ (see e.g.~\cite[Proposition~6.4.12]{AK16}), by~\cite{Bal90} we know that  $\MM$ is isometric to a subset of $\ell_p^{n(n-1)/2}$. Hence, thanks to Theorem~\ref{thm:transition} for arbitrarily large $n\in \N$ there is an $n$-point subset of $\ell_p$ such that if it embeds into $\ell_p^m$ with distortion $O(1)$, then necessarily: 
\begin{equation}\label{eq:JL lower}
m\gtrsim_p \bigg(\frac{\log n}{(\log\log n)^2}\bigg)^{\frac{p}{2}}.
\end{equation}

Since $p>2$, the right hand side of~\eqref{eq:JL lower} is asymptotically larger  as $n\to \infty$ than the $O(\log n)$ dimension bound of the JL lemma, though only mildly so, and it remains a major challenge to establish any nontrivial metric dimension reduction result for $n$-point subsets of $\ell_p$ when $2<p<\infty$; a  discussion of such questions can be found in the survey~\cite{Naor-ribe}.  Even though~\eqref{eq:JL lower}  is merely a modest counterexample to the ``vanilla'' extension of the JL-lemma to $\ell_p$ when $2<p<\infty$, it answers an open question that was raised in multiple venues. For examples, this matter was broached in~\cite[Problem~3.2]{Mat-problems} and~\cite[Question~13]{Naor-ribe}, as well as~\cite[Problem~5.2]{Mat-problems}, where a concrete strategy towards an impossibility result was proposed, but that strategy has been subsequently ruled out in~\cite{CK05}. 

If $1<p<2$, then it remains open to determine whether every $n$-point subset of $\ell_p$ embeds with distortion $O(1)$ into $\ell_p^m$ for $m\lesssim \log n$, though we conjecture that the answer is negative also for $p$ in this range. The cases $p=1$ and $p=\infty$ have been settled negatively in~\cite{BC03} (see also~\cite{LeeNaor,Regev} for different proofs, as well as~\cite{ACNN} for a strengthening) and~\cite{AR92,Mat96,RV,LMN,Rab08,Nao17}, respectively.  The embedding of the JL lemma is actually a linear mapping, and if one adds such linearity as a requirement, then it is known~\cite{CS02,LMN} that the JL lemma fails to hold in $\ell_p$ for any $p\neq 2$. In fact, by~\cite{JN09} the statement of the JL lemma---with linearity of the embedding and its target being a subspace of logarithmic dimension---almost characterizes (in a precise sense that is formulated in~\cite{JN09}) Hilbert spaces among Banach spaces, yet there exist Banach spaces with this favorable dimension reduction property that are not isomorphic to a Hilbert space (nevertheless, their finite dimensional subspaces are very close to Euclidean spaces in the sense that the growth of their distance to Euclidean space in terms of their dimension is of inverse Ackermann type).

Given $K\in \N$, a metric space $(\MM,d_\MM)$ is said to be $K$-doubling if for every $x\in \MM$ and  $r>0$ there are $y_1,\ldots,y_K\in \MM$ such that $B_\MM(x,2r)\subset B_\MM(y_1,r)\cup\ldots\cup B_\MM(y_K,r)$, where $B_\MM(z,\rho)=\{y\in \MM:\ d_\MM(x,y)\le \rho\}$ is the closed ball of radius $\rho\ge 0$ centered at $z\in \MM$.\footnote{Below we will also use the standard notation $B_\MM^\circ(z,\rho)=\{y\in \MM:\ d_\MM(x,y)< \rho\}$  for the corresponding open ball.} Dimension reduction for doubling metric spaces is also a topic of great interest that exhibits major longstanding mysteries; see e.g.~\cite[Section~3]{Naor-ribe}. In particular, the  Lang--Plaut problem~\cite{LP01} asks if every $K$-doubling subset of $\ell_2$ embeds with distortion $O_K(1)$ into $\ell_2^m$ for $m=O_K(1)$. In terms of the dependence on the doubling constant, a full analog of the JL lemma would be that every $K$-doubling subset of $\ell_2$ embeds with distortion $O(1)$ into $\ell_2^m$ for $m\lesssim \log K$; we call this possibility the { sharp Lang--Plaut problem}. 

The works~\cite{BGN14,LN} ruled out a natural $\ell_p$ analog of the Lang--Plaut problem when $2<p<\infty$ by demonstrating that there are $O(1)$-doubling subsets of $\ell_p$ that do not embed (with any distortion) into $\ell_p^m$ for any $m\in \N$. In fact, those examples do not even embed (with any distortion) into $L_q$ for any $1<q<p$, and by~\cite{NY17} the example of~\cite{LN} also does not embed into $L_1$ (see~\cite[Remark~8]{NY18}).

By~\cite{Assouad}, for $0<\theta<1$ the $\theta$-snowflakes of the aforementioned counterexamples of~\cite{BGN14,LN} embed with distortion $O_{\theta}(1)$ into $\ell_2^m$ for $m=O_\theta(1)$, so~\cite{BGN14,LN} do not rule out  the $\ell_p$ analog of the Lang--Plaut problem for snowflakes of doubling metric spaces.  More generally,  by~\cite{Assouad} the $\theta$-snowflake of any $K$-doubling metric space embeds with distortion $O_{K,\theta}(1)$ into $\ell_2^m$ for $m=O_{K,\theta}(1)$; see~\cite{NN} for a better estimate on the dimension here as $\theta\to 1^-$ (see also~\cite{DS13} for a different proof). Thus, the ``vanilla'' Lang--Plaut problem for snowflakes of doubling metric spaces has a positive answer, even if they are not assumed  to be subsets of a Hilbert space. Asking in this context for the distortion to be $O(1)$ is, however,  an entirely different matter that has been  investigated in~\cite{HM05,ABN08,GK,BRS,Nei16,BG16}, in part due to the algorithmic implications of  small distortion low-dimensional embeddings of snowflaked doubling  metric spaces.  The Lang--Plaut problem for $\frac13$-snowflakes\footnote{The ensuing discussion extends mutatis mutandis to $\theta$-snowflakes for any $0<\theta<1$, but for illustration purposes within the present introductory discussion it is beneficial to carry less parameters by choosing one specific, but arbitrary, exponent.} asks whether for every $K\in \N$, the $\frac13$-snowflake of any $K$-doubling subset of $\ell_2$ embeds with distortion $O(1)$ into $\ell_2^m$ for $m=O_{K}(1)$, and correspondingly the sharp version thereof asks if one could achieve $m\lesssim \log K$ here. The former question was answered positively: it was proved in~\cite{GK}  with $m\lesssim (\log K)^2$ and~\cite{BRS} proved it with $m\lesssim (\log K)\log\log K$, thus even coming close to answering positively the sharp Lang--Plaut problem for snowlakes. Since for every integer $k\ge 4$ a standard volumetric covering argument shows that $S^{k-1}$ is $K$-doubling for some $K\in \N$ with $\log K\asymp k$,  Theorem~\ref{thm:transition}  shows that the analog of~\cite{BG,BRS} for embeddings into $\ell_p^m$ fails if $p$ is large enough; to state a concrete special case, for, say, $p=90$, we get that for arbitrarily large $K\in \N$ there exists a $K$-doubling subset of $\ell_2$ such that if its $\frac13$-snowflake embeds with distortion $O(1)$ into $\ell_{90}^m$, then necessarily $m\gtrsim ((\log K)/\log\log K)^{15}$. While an examination of the proofs in~\cite{BG,BRS} immediately  reveals that they use the fact that the target space is Euclidean, we thus see that some Euclidean reasoning is inherently needed there.

\subsection{On the proof of Theorem~\ref{thm:transition}} Fix $2<p<\infty$ and $m,k\in \N$. By~\cite{PR75,BDGJN77}, if $\ell_2^k$ embeds with distortion $O(1)$ into $\ell_p^m$, then necessarily $m\gtrsim_p k^{p/2}$, which is sharp by~\cite{FLM}. Thus, the exponent $p/2$ in the case $\theta=1$ of~\eqref{eq:snow-lower} is not surprising, and a natural approach to proving~\eqref{eq:snow-lower} would be to suitably discretize the proofs in~\cite{PR75,BDGJN77} (which are conceptually different from each other). However, we do not see how to proceed in that fashion.  To deduce the aforementioned bi-Lipschitz non-embeddability result from~\cite{PR75,BDGJN77}, one first differentiates the embedding (using Rademacher's theorem~\cite{Rad19} on almost everywhere differentiability of Lipschitz functions) to obtain a linear mapping with the same distortion guarantee, and then it is possible to perform  the  ingenious  reasoning  in~\cite{PR75,BDGJN77}, which relies highly on the linear theory. Furthermore, treating $\theta$-snowflakes as in Theorem~\ref{thm:transition} would require a different strategy, as H\"older functions need not have any point of differentiability, whence it is unclear if a reduction to the linear theory is possible. 

Due to the above, the ensuing proof of~\eqref{eq:snow-lower}  uses an inherently nonlinear approach that is entirely different from the reasoning in~\cite{PR75,BDGJN77} and applies painlessly to H\"older functions.  The conceptual origin of this (short) proof is the work~\cite{RV} that uses concentration of measure to restrict bi-Lipschitz embeddings of L\'evy families of metric measure spaces into $\ell_\infty^m$, and to deduce interesting linear consequences of the existence of such embeddings.  In fact, our approach is closest to the implementation of this idea in the proof of~\cite[Proposition~4.2]{LMN} (see also~\cite{Rab08,Nao14}). All in all, the proof of~\eqref{eq:snow-lower} is quite short and conceptual, and the key new idea is the mere realization that it could be used for treating dimension reduction in $\ell_p$ by trivially bounding the $\ell_p^m$ norm by $m^{1/p}$ times the $\ell_\infty^m$ norm.     

We derive the  upper bound in Theorem~\ref{eq:snow-lower} by combining (substantial) embedding methods in the literature. Specifically, small distortion embeddings of the $\theta$-snowflake of the real line into low-dimensional Euclidean space is a classical subject that has been  investigated in e.g.~\cite{Gla58,Kahane,Assouad,Talagrand,DH99,HM12}. One can follow the strategy of~\cite{Assouad}  together with its enhancement herein (see Proposition~\ref{prop:Xcubed} below) to (sharply) improve its target dimension, showing  that for every $p\ge 1$ the $\theta$-snowflake of $\R$ embeds with distortion $O(1)$ into $\ell_p^m$, where $m\lesssim 1/\theta$.  We pass from this to the embedding of the $\theta$-snowflake of $\ell_2^k$ into $\ell_p^m$ of asymptotically optimal  dimension, thus establishing the sharp threshold phenomenon of Theorem~\ref{thm:transition}, by combining it through a short case analysis with the linear embeddings of~\cite{FLM,Kas77,Vyb08,FPRU10}.

\section{Proof of Theorem~\ref{thm:transition}}\label{sec:lower}

The proof of the lower bound~\eqref{eq:snow-lower} on  the target dimension $m$ of Theorem~\ref{thm:transition} relies on the following statement, 
which is inspired by \cite[Proposition~4.2]{LMN}:
\begin{lemma}\label{prop:engine} Fix $k,m\in \N$, as well as $0\le \theta\le 1$ and $\tau>0$. For each $i\in \{1,\ldots,m\}$, suppose that $\varphi_i : S^{k-1} \to \R$ is $\theta$-H\"older with constant $1$ with respect to the $\ell_2^k$ metric.   Suppose furthermore that:
\begin{equation}\label{eq:antipodal}
  \forall x\in S^{k-1},\qquad   \max_{i\in \{1,\ldots,m\}}\abs{\varphi_i(x)-\varphi_i(-x)}\ge2\tau.
\end{equation}
Then,
\begin{equation}\label{eq:use concentration}
(k-1)\tau^{\frac{2}{\theta}}\le 2\log m+4.
\end{equation}
\end{lemma}

\begin{proof} For each $i\in \{1,\ldots,m\}$, let $\mu_i$ be a median of $\f_i$, i.e., the $\sigma_{k-1}$-measure of   $\{x\in S^{k-1}:\ \varphi_i(x)\le\mu_i\}$ and $\{x\in S^{k-1}:\ \varphi_i(x)\ge\mu_i\}$ is at least $1/2$. As explained in~\cite[Chapter~2]{MS86} (in the context of more general moduli of continuity),  the $\theta$-H\"older assumption on $\f_i$ combined with the isomperimetric inequality for $\sigma_{n-1}$ gives the following estimate: 
\begin{equation}\label{eq:coord-conc}
 \forall i\in \{1,\ldots,m\},\qquad    \sigma_{k-1}\Big(\big\{x\in S^{k-1}:\ |\varphi_i(x)-\mu_i|\ge \tau\big\}\Big)\le 
    \sqrt{2\pi}e^{-\frac{k-1}{2}\tau^{\frac{2}{\theta}}}.
\end{equation}
Hence, $\sigma_{k-1}(E)\le \sqrt{2\pi}m e^{-\frac{k-1}{2}\tau^{\frac{2}{\theta}}}$, where we denote:
$$
E\eqdef \Big\{x\in S^{k-1}:\max_{i\in \{1,\ldots,m\}}|\varphi_i(x)-\mu_i|\ge\tau\Big\}=\bigcup_{i=1}^m \big\{x\in S^{k-1}:\ |\varphi_i(x)-\mu_i|\ge \tau\big\}.
$$
At the same time, the hypothesis~\eqref{eq:antipodal} implies  $E\cup(-E)=S^{k-1}$. Consequently,  $\sigma_{k-1}(E)\ge1/2$. By contrasting this with the aforementioned upper bound on $\sigma_{k-1}(E)$, we arrive at the desired conclusion~\eqref{eq:use concentration}.
\end{proof}

Using Lemma~\ref{prop:engine}, we can now prove the first part~\eqref{eq:snow-lower} of Theorem~\ref{thm:transition}.

\begin{proof}[Proof of~\eqref{eq:snow-lower}] If $m\ge (4D)^{-p}k^{\theta p/2}$, then~\eqref{eq:snow-lower} already holds, so we may assume from now that: 
\begin{equation}\label{eq:m small}
m<\frac{k^{\frac{\theta p}{2}}}{(4D)^p}.
\end{equation}

Suppose that $\psi:\MM\to \ell_p^m$ satisfies $\|x-y\|_2^\theta \le \|\psi(x)-\psi(y)\|_p\le D\|x-y\|_2^\theta$ for every $x,y\in \MM$. As $\|x\|_\infty\le \|x\|_p\le m^{1/p}\|x\|_\infty$ for every $x\in \R^m$, by rescaling $\psi$ we get a function $\f:\MM\to \R^m$ that satisfies:
\begin{equation}\label{eq:pass to infty}
\forall x,y\in \MM,\qquad \frac{1}{Dm^{\frac{1}{p}}}\|x-y\|_2^\theta\le \|\f(x)-\f(y)\|_\infty\le \|x-y\|_2^\theta.
\end{equation}
By the nonlinear Hahn--Banach theorem (see~\cite{Mac34}, or e.g.~\cite[Chapter~1]{BL00}), there exists  $\Phi:\R^k\to \R^m$ whose restriction to $\MM$ coincides with $\f$, and it is $1$-Lipschitz as a mapping from $\ell_2^k$ to $\ell_\infty^m$. As $\MM$ is assumed to be $(1/\sqrt{k})$-dense in $S^{k-1}$, for every $x\in S^{k-1}$ there are $y,z\in \MM$ such that $\|x-y\|_2\le 1/\sqrt{k}$ and $\|x+z\|_2\le 1/\sqrt{k}$. Thus, $\|y-z\|_2\ge 2-2/\sqrt{k}\ge 1$, as $k\ge 4$. Consequently:
\begin{multline*}
\norm{\Phi(x) - \Phi(-x)}_\infty \ge \norm{\Phi(y) - \Phi(z)}_\infty - \norm{\Phi(x) - \Phi(y)}_\infty - \norm{\Phi(-x) - \Phi(z)}_\infty\\ \stackrel{\eqref{eq:pass to infty}}{\ge} \frac{1}{Dm^{\frac{1}{p}}} \norm{y-z}_2^\theta-\|x-y\|_2^\theta-\|x+z\|_2^\theta 
\ge \frac{1}{Dm^{\frac{1}{p}}} - \frac{2}{k^{\theta/2}} \stackrel{\eqref{eq:m small}}{\ge} \frac{1}{2Dm^{\frac{1}{p}}} .
\end{multline*}
Now, Lemma~\ref{prop:engine}  applied to the coordinates of $\Phi$ provides the following estimate, which simplifies to~\eqref{eq:snow-lower}:
\begin{equation*}
\frac{k}{2}\bigg(\frac{1}{4Dm^{\frac{1}{p}}}\bigg)^{\frac{2}{\theta}}<(k-1)\bigg(\frac{1}{4Dm^{\frac{1}{p}}}\bigg)^{\frac{2}{\theta}}\le 2\log m+4\stackrel{\eqref{eq:m small}}{<}\theta p \log k +4.\tag*{\qedhere}
\end{equation*}
\end{proof}

\smallskip

We will next pass to the proof of the second part of Theorem~\ref{thm:transition}. One of its inputs is the following result:

\begin{prop}\label{prop:Xcubed} For every $K\in \N$ there are $A_K,D_K\ge 1$ such that if $0<\theta\le 3/5$ and $\X$ is a normed space with $\dim \X\ge A_K/\theta$, then the $\theta$-snowflake of any $K$-doubling metric space embeds into $\X^3$ with distortion $D_K$. 
\end{prop}

In Proposition~\ref{prop:Xcubed}, as well as throughout what follows,  we use the following notation and conventions. Given a normed space  $(\X,\|\cdot\|_\X)$, its unit ball is $B_\X=\{x\in \X:\ \|x\|_\X\le 1\}$. The norm on $\X^3$ will always be tacitly assumed  to be  such that $B_{\X^3}=B_\X^3$, i.e.,  $\|(x,y,z)\|_{\X^3}=\max\{\|x\|_\X,\|y\|_\X,\|z\|_\X\}$, for every $(x,y,z)\in \X^3$.  

A straightforward inspection of the ensuing proof of Proposition~\ref{prop:Xcubed}   reveals that what it actually requires is that $2/3-\theta=\Omega(1)$. We chose the arbitrary value $3/5$ for concreteness; replacing it by any fixed  quantity that is strictly smaller than $2/3$ influences only the implicit universal constants.  

Proposition~\ref{prop:Xcubed} is reminiscent of Assouad's embedding theorem~\cite{Assouad}, which implies the same result without the stated dependence of the dimension on $\theta$. The ensuing proof of Proposition~\ref{prop:Xcubed} follows the strategy of~\cite{Assouad} with a twist that relies on Lemma~\ref{lem:one sided JL in any norm} below to obtain dimension dependence on $\theta$ that is  better than what comes from the proof in~\cite{Assouad} (which yields an exponentially worse bound). That dependence is optimal in terms of its dependence on $\theta$ as $\theta\to 0^+$ for fixed $K$, which is what we need herein as we will use Proposition~\ref{prop:Xcubed} only when the embedded metric space is the real line, whence $K=O(1)$. Nevertheless, understanding in this context the dependence on $K$ as $K\to \infty$ is interesting; see Section~\ref{sec:sharp assouad} below. 


\begin{lemma}\label{lem:one sided JL in any norm} Fix $C,s\in \N$. If $(\X,\|\cdot\|_\X)$ is a normed space with $\dim \X\ge 3s\log (eC)$, then there are  $v_1,\ldots,v_C\in B_\X$ such that for every $\upxi_1,\ldots,\upxi_C\in \R$ and every $S\subset \{1,\ldots,C\}$ with $|S|=s$ we have:
\begin{equation}\label{eq:1/2 all subsets}
  \Big\|\sum_{j\in S} \upxi_j v_j\Big\|_\X\ge \frac12\max_{j\in S}|\upxi_j|, 
\end{equation}
\end{lemma}

\begin{proof}  The desired conclusion coincides with requiring that  $d_\X(v_i,\mathrm{span}(\{v_j\}_{j\in S\setminus\{i\}}))\ge 1/2$ for every subset $S$ of  $\{1,\ldots,C\}$  of size $s$ and every $i\in S$, where we write $d_\X(x,\sub)=\inf_{y\in \sub} \|x-y\|_\X$ for  $x\in \X$ and $\sub\subset \X$.

Set $n=\dim \X$. Fix an $n$-dimensional  Lebesgue measure $\vol_n$ on $\X$, normalized so that $\vol_n(B_\X)=1$. If $k\in \{0,\ldots,n\}$ and $\Y\subset \X$ is a $k$-dimensional linear subspace of $\X$, then the following estimate holds:
\begin{equation}\label{eq:neighborhood of subspace}
\forall \d>0,\qquad \vol_n\big[x\in B_\X:\ d_\X(x,\Y)\le \d\big]\le \tbinom{n}{k}\d^{n-k}. 
\end{equation}
Indeed, fixing a Euclidean structure on $\X$ that induces $\vol_n$, letting $\Proj_\Y^\perp:\X\to \Y^\perp$ denote the orthogonal projection onto the orthogonal complement $\Y^\perp$ of $\Y$, and denoting for each $0\le d\le n$ the $d$-dimensional Hausdorff measure that $\vol_n$ induces on $\X$ by $\vol_d$, by Fubini we have:
\begin{align}\label{eq:use rogers shephard}
\begin{split}
\vol_n\big[x\in B_\X:\ &d_\X(x,\Y)\le \d\big]=\int_{\d\Proj_{\Y^\perp}B_\X}\vol_{k} \big(B_\X\cap (z+\Y)\big)\ud z\\&\le 
\Big(\max_{z\in \d\Proj_{\Y^\perp}B_\X} \vol_{k} \big(B_\X\cap (z+\Y)\big)\Big)\d^{n-k}\vol_{n-k}\big(\Proj_{\Y^\perp}B_\X\big)\le \tbinom{n}{k}\d^{n-k}\vol_n(B_\X)=\tbinom{n}{k}\d^{n-k}, 
\end{split}
\end{align}
where the penultimate step of~\eqref{eq:use rogers shephard} is the Rogers--Shephard projection–section inequality~\cite[Theorem~1]{RS58}. 

We conclude  by considering as follows i.i.d.~random elements  $v_1,\ldots,v_C$ of $B_\X$ that are distributed according to the restriction of $\vol_n$ to $B_\X$: 
\begin{align}\label{eq:the prob estimate}
\begin{split}
&\vol_n^{\otimes C} \Big[(v_1,\ldots,v_C)\in B_\X^C:\ \min_{\substack{ S\subset \{1,\ldots, C\}\\ |S|=s\\i\in S}} d_\X\big(v_i,\mathrm{span}(\{v_j\}_{j\in S\setminus \{i\}})\big)\ge \frac12 \Big]\\&\ \ \ \le 1- \sum_{\substack{S\subset \{1,\ldots,C\}\\|S|=s}} \sum_{i\in S}\vol_n^{\otimes C} \Big[(v_1,\ldots,v_C)\in B_\X^C:\ d_\X\big(v_i,\mathrm{span}(\{v_j\}_{j\in S\setminus\{i\}})\big)<\frac12 \Big]\stackrel{\eqref{eq:use rogers shephard}}{\ge} 1- \frac{\tbinom{C}{s}s\tbinom{n}{s-1}}{2^{n-s+1}}>0, 
\end{split}
\end{align}
where the last step of~\eqref{eq:the prob estimate} is a straightforward exercise using the assumption $n\ge 3s\log (eC)$.  
\end{proof}

\begin{remark} {\em It is trivial to adjust the proof of Lemma~\ref{lem:one sided JL in any norm} to get its version with the factor $\frac12$ in~\eqref{eq:1/2 all subsets} replaced by $1-\e$ for any $0<\e<1$; we omit the details as this is not needed herein and the modification  is mechanical. }
\end{remark}

The following lemma in the next step toward Proposition~\ref{prop:Xcubed}; it mimics the reasoning in~\cite{Assouad} (notably, the coloring argument therein) while using  the vectors of Lemma~\ref{lem:one sided JL in any norm} in place of an orthonormal basis.

\begin{lemma}\label{lem:one scale} Fix $K,\ell\in \N$. For every normed space $(\X,\|\cdot\|_\X)$  such that $\dim \X\ge 8\ell K^2\log (eK)$, and for every $K$-doubling metric space $(\MM,d_\MM)$, there exists a mapping $\phi:\MM\to \X$ satisfying:
\begin{equation}\label{scale 1}
\forall x,y\in \MM,\qquad \frac{1}{8K^2}\1_{\left\{4\le d_\MM(x,y)< 2^{\ell}-4\right\}}\le \|\phi(x)-\phi(y)\|_\X\le \min\left\{d_\MM(x,y), 1\right\}. 
\end{equation}
\end{lemma}

\begin{proof} Fix $\NN\subset \MM$ that is maximal with respect to inclusion relative to the requirement that every distinct $a,b\in \NN$ satisfy $d_\MM(a,b)>1$. Thus, for each $x\in \MM$ we may fix a point $a_x\in \NN$ satisfying $d_\MM(a_x,x)\le 1$.   As $\{B_\MM(x,1/2)\}_{a\in \NN}$ are pairwise disjoint, by iterating the $K$-doubling assumption we know that: 
\begin{equation}\label{eq:small balls in net}
\forall s\in \N,\qquad \max_{x\in \MM}|\NN\cap B_\MM(x,2^{s-1})|\le K^{s}.
\end{equation}
As explained in~\cite{Assouad} (see also the exposition in e.g.~the monograph~\cite[Chapter~12]{Hei01}), if we set $C= K^{\ell+1}+1$, then the case $s=\ell+1$ of~\eqref{eq:small balls in net} implies that there exists a function (a coloring)  $\upchi:\NN\to \{1,\ldots, C\}$ such that:
\begin{equation}\label{eq:coloring coindition}
\forall a,b\in \NN,\qquad 0<d_\MM(a,b)\le 2^{\ell}\implies \upchi(a)\neq\upchi(b). 
\end{equation}

We may assume from now that $\ell\ge 4$ as if $\ell\in \{1,2,3\}$, then~\eqref{scale 1} holds even when  $\phi\equiv 0$.  Observe that $\dim\X\ge 8\ell K^2\log(eK)\ge 6K^2\log(eC)$, where the second inequality is elementary calculus, using $\ell\ge 4$. Thus,  the assumption of Lemma~\ref{lem:one sided JL in any norm} holds with $C$ as above and  $s=2K^2$, whence we can fix $v_1,\ldots,v_C\in B_\X$  that have the following property:
\begin{equation}\label{eq:super 1/2 bessel}
\forall \upxi_1,\ldots,\upxi_C\in \R,\ \forall S\subset \{1,\ldots,C\},\qquad |S|=2K^2\implies \Big\|\sum_{j\in S} \upxi_j v_j\Big\|_\X\ge \frac12 \max_{j\in S}|\upxi_j|. 
\end{equation}
Using these vectors, we  now define  $\phi:\MM \to \R^n$ as follows:
\begin{equation}\label{eq:def our f_i}
\forall x\in \MM,\qquad \phi(x)\eqdef  \sum_{a\in \NN\cap B_\MM^\circ(x,2)} \frac{2-d_\MM(x,a)}{4K^2} v_{\upchi(a)}.
\end{equation}

If  $x,y\in \MM$ satisfy $d_\MM(x,y)\ge 4$, then $B_\MM^\circ(x,2)\cap B_\MM^\circ(y,2)=\emptyset$, so the index set $\NN\cap B_\MM^\circ(x,2)$ in the sum in~\eqref{eq:def our f_i}  is disjoint from the index set $\NN\cap B_\MM^\circ(y,2)$ in the corresponding sum in~\eqref{eq:def our f_i} with $x$ replaced by $y$. If also $d_\MM(x,y)\le 2^{\ell}-4$, then $0<d_\MM(a,b)< 2^{\ell}$ for  $(a,b) \in  B_\MM^\circ(x,2)\times B_\MM^\circ(y,2)$, whence thanks to~\eqref{eq:coloring coindition} we know that $\upchi(a)\neq \upchi(b)$ for every $a\in \NN\cap B_\MM^\circ(x,2)$ and $b\in \NN\cap B_\MM^\circ(y,2)$. Furthermore, $0<d_\MM(a,b)<4\le 2^\ell$ if   $\{a,b\}\subset \NN\cap B_\MM^\circ(x,2)$ or $\{a,b\}\subset \NN\cap B_\MM^\circ(x,2)$ and $a\neq b$, so  $\upchi(a)\neq \upchi(b)$ by~\eqref{eq:coloring coindition}. This shows that the ``active colors'' $\{\upchi(a):\ a\in \NN\cap (B_\MM^\circ(x,2)\cup B_\MM^\circ(y,2))\}$ are distinct, and by~\eqref{eq:small balls in net} with $s=2$ there are at most $2K^2$  such colors. By the definition~\eqref{eq:def our f_i} of $\phi$ and~\eqref{eq:super 1/2 bessel} we therefore have the following lower bound: 
$$
\|\phi(x)-\phi(y)\|_\X \ge  \max_{a\in \NN\cap \left(B_\MM^\circ(x,2)\cup B_\MM^\circ(y,2)\right)} \frac{|2-d_\MM(a_,x)|}{8K^2}
\ge  \frac{2-d_\MM(a_x,x)}{8K^2} \ge \frac{1}{8K^2}.
$$
This proves the first inequality in~\eqref{scale 1}. The second inequality in~\eqref{scale 1} is much simpler to justify, because it can be deduced as follows using only $\|v_1\|_\X,\ldots,\|v_C\|_\X\le 1$, the triangle inequality for $\|\cdot\|_\X$, and the fact that the function  $(x\in \MM)\mapsto  \max\{2-d_\MM(x,a),0\}$ is $1$-Lipschitz and bounded by $2$: 
\begin{align*}
\|\phi(x)-\phi(y)\|_\X\stackrel{\eqref{eq:def our f_i}}{=}\bigg\|&\sum_{a\in \NN\cap \left(B_\MM^\circ(x,2)\cup B_\MM^\circ(y,2)\right)}\frac{\max\big\{2-d_\MM(x,a),0\big\}-\max\big\{2-d_\MM(y,a),0\big\}}{4K^2} v_{\upchi(a)}\bigg\|_\X\\&\ \ \ \ \ \ \ \ \ \le 
\frac{\max\big\{d_\MM(x,y),2\big\}}{4K^2}\left(|\NN\cap B_\MM(x,2)|+|\NN\cap B_\MM(y,2)|\right)\stackrel{\eqref{eq:small balls in net}}{\le} \max\big\{d_\MM(x,y),1\big\}.\tag*{\qedhere}
\end{align*}
\end{proof}

The following consequence of Lemma~\ref{lem:one scale}  applies it to a sequence of multiples of the original metric:

\begin{corollary}\label{cor:all scales} Fix  $K\in \N$ and $Q>1$. Let $(\X,\|\cdot\|_\X)$ be a normed space with $\dim\X\ge 32K^2(1+\log K)\log(eQ)$, and let $(\MM,d_\MM)$ be a $K$-doubling metric space. Then,  there is a sequence of mappings   $\{\phi_i:\MM\to \X\}_{i\in \Z}$ such that the following estimates hold for every distinct $x,y\in \MM$ and every $i\in \Z$: 
\begin{equation}\label{eq:log version of pairwise in MM}
   \left\{\begin{array}{ll} \|\phi_i(x)-\phi_i(y)\|_\X\le \min\big\{d_\MM(x,y), Q^i\big\},\\
\|\phi_{\lfloor \log_Q d_\MM(x,y)\rfloor}(x)-\phi_{\lfloor \log_Q d_\MM(x,y)\rfloor}(y)\|_\X\ge \frac{1}{32K^2} Q^{\lfloor \log_Q d_\MM(x,y)\rfloor}. \end{array}\right.
\end{equation}
\end{corollary}

\begin{proof} Denote $\ell=\ell(Q)=\lceil \log_2(Q+1)\rceil+2\in \N$.   Then, $\dim\X\ge 32K^2(1+\log K)\log(eQ)\ge 8\ell K^2\log(eK)$, where the second inequality is an elementary calculus exercise using $Q\ge 1$. Thus, the assumption of Lemma~\ref{lem:one scale} holds.  For every $i\in \Z$, by applying  Lemma~\ref{lem:one scale}  to the metric space $(\MM,4Q^{-i}d_\MM)$, which is still $K$-doubling, and then multiplying the resulting mapping by $Q^i/4$, we obtain $\phi_i:\MM\to \X$ that satisfies:  
\begin{equation*}\label{eq:ell to Q}
\forall x,y\in \MM,\qquad \frac{Q^i}{32K^2}\1_{\left\{Q^i\le d_\MM(x,y)< (2^{\ell-2}-1)Q^i\right\}}\le \|\phi_i(x)-\phi_i(y)\|_\X\le \min \Big\{d_\MM(x,y), \frac14 Q^i\Big\}. 
\end{equation*}
This implies~\eqref{eq:log version of pairwise in MM} because $2^{\ell-2}-1\ge Q$ by the above choice of $\ell$. 
\end{proof}

 The following lemma is a standard bookkeeping fact for lacunary  superposition of mappings, variants of which are used in~\cite{Assouad} and many other places in the literature. As we could not locate a reference where it is stated  with the dependence on parameters that we use herein, we will provide its simple proof.

\begin{lemma}\label{fact:abstract} Let  $\fS$ be   a set and suppose that  $L:\fS^2\to [0,\infty)$ satisfies the following requirements:
$$
\forall s,t\in \fS,\qquad L(s,t)=L(t,s)\qquad\mathrm{and}\qquad   L(s,t)=0\iff s=t.
$$
Fix  $0<\beta,\theta\le 1$   and $Q>1$ for which:\footnote{The proof of Lemma~\ref{fact:abstract}  requires $2/3-\theta=\Omega(1)$; the specific upper bound on $\theta$ in~\eqref{Q large theta small} is fixed here for concreteness only.} 
\begin{equation}\label{Q large theta small}
Q\ge \left(\frac{8}{\beta}\right)^{\frac{3}{\theta}}\qquad \mathrm{and}\qquad  \theta\le \frac35. 
\end{equation}
Let $(\X,\|\cdot\|_\X)$ be a Banach space such that for every  $i\in \Z$ there exists a mapping $\phi_i:\fS\to \X$ such that the following estimates hold for every distinct $s,t\in \fS$:
\begin{equation}\label{eq:log version of pairwise}
 \|\phi_i(s)-\phi_i(t)\|_\X\le \min\big\{L(s,t), Q^i\big\}\qquad\mathrm{and}\qquad 
\|\phi_{\lfloor \log_Q L(s,t)\rfloor}(s)-\phi_{\lfloor \log_Q L(s,t)\rfloor}(t)\|_\X\ge \beta Q^{\lfloor \log_Q L(s,t)\rfloor}. 
\end{equation}
Then, there exists a mapping $F:\fS\to \X^3$ that satisfies: 
\begin{equation}\label{eq:F bookkeping guarantees}
\forall s,t\in \fS,\qquad \frac{\beta}{Q^\theta} L(s,t)^{\theta}\lesssim  \|F(s)-F(t)\|_{\X^3}\lesssim Q^\theta L(s,t)^{\theta}. 
\end{equation}
\end{lemma}

\begin{proof} Fix $s_0\in \fS$ and define $f_1,f_2,f_3:\fS\to \X$ as follows: 
\begin{equation}\label{eq:def one of 3 coordinates}
\forall r\in \{1,2,3\},\ \forall s\in \fS,\qquad f_r(s)\eqdef \sum_{m\in \Z} \frac{1}{Q^{(1-\theta)(3m+r)}}\big(\phi_{3m+r}(s)-\phi_{3m+r}(s_0)\big).
\end{equation}
Because~\eqref{Q large theta small} implies that $Q> 3$, the first inequality in~\eqref{eq:log version of pairwise} ensures that the  series  in~\eqref{eq:def one of 3 coordinates} converges absolutely   (at a geometric rate). We can therefore define $F:\fS\to \X^3$ by:
\begin{equation}\label{eq:defF}
\forall s\in \fS,\qquad F(s)\eqdef \big(f_1(s),f_2(s),f_3(s)\big).
\end{equation}

For every distinct $s,t\in  \fS$ let $\lambda_{st}\eqdef \lfloor \log_Q L(s,t)\rfloor$, i.e., $\lambda_{st}$ is the unique integer that satisfies:
\begin{equation}\label{eq:L beteen powers of Q}
Q^{\lambda_{st}}\le L(s,t)<Q^{\lambda_{st}+1}.
\end{equation}
With this notation, the second inequality in~\eqref{eq:log version of pairwise} becomes $\|\phi_i(s)-\phi_i(t)\|_2\ge \beta Q^{\lambda_{st}}$. To demonstrate the first inequality in~\eqref{eq:F bookkeping guarantees}, divide with remainder modulo $3$ to get $\mu_{st}\in \Z$ and $\rho_{st}\in \{1,2,3\}$ satisfying: 
\begin{equation}\label{eq:def mod 3}
\lambda_{st}=3\mu_{st}+\rho_{st}.
\end{equation}
The desired bound is seen by considering as follows the contribution of the $\rho_{st}$ coordinate of $F$: 
\begin{align}\label{eq:lower mod 3}
\begin{split}
\|F(s)-F(t)\|_2&\ge \|f_{\rho_{st}}(s)-f_{\rho_{st}}(t)\|_2\\&\!\!\!\!\!\!\!\!\!\!\!\!\!\!\!\!\!\!\!\stackrel{\eqref{eq:def one of 3 coordinates} \wedge \eqref{eq:defF}\wedge \eqref{eq:def mod 3}}{\ge} \frac{\|\phi_{\lambda_{st}}(s)-\phi_{\lambda_{st}}(t)\|_2}{Q^{(1-\theta)\lambda_{st}}} -\sum_{m\in \Z\setminus \{\mu_{st}\}}
\frac{\|\phi_{\lambda_{st}+3(m-\mu_{st})}(s)-\phi_{\lambda_{st}+3(m-\mu_{st})}(t)\|_2}{Q^{(1-\theta)\left(\lambda_{st}+3(m-\mu_{st})\right)}}\\
&\!\!\!\!\!\!\!\!\!\!\!\stackrel{\eqref{eq:log version of pairwise}\wedge\eqref{eq:L beteen powers of Q}}{\ge} \beta Q^{\theta \lambda_{st}}-\sum_{m=-\infty}^{\mu_{st}-1}Q^{\theta\left(\lambda_{st}+3(m-\mu_{st})\right)}
-\sum^{\infty}_{m=\mu_{st}+1}\frac{Q^{\lambda_{st}+1}}{Q^{(1-\theta)\left(\lambda_{st}+3(m-\mu_{st})\right)}}\\&= \bigg(\beta-\frac{1}{Q^{3\theta}-1}-\frac{Q}{Q^{3(1-\theta)}-1}\bigg)Q^{\theta\lambda_{st}}\stackrel{\eqref{Q large theta small} \wedge \eqref{eq:L beteen powers of Q}}{\gtrsim} \frac{\beta}{Q^\theta} L(s,t)^\theta,
\end{split}
\end{align}
 where the last step of~\eqref{eq:lower mod 3} is a straightforward  calculus exercise. The proof of the rest of~\eqref{eq:F bookkeping guarantees} is simpler: 
\begin{align*}
\|F(s)-F(t)\|_2&\stackrel{\eqref{eq:def one of 3 coordinates} \wedge \eqref{eq:defF}}{\le} \sum_{r=1}^3 \sum_{m\in \Z} \frac{\|\phi_{3m+r}(s)-\phi_{3m+r}(t)\|_2}{Q^{(1-\theta)(3m+r)}}\\
&\stackrel{\eqref{eq:log version of pairwise}\wedge\eqref{eq:L beteen powers of Q}}{\le}  \sum_{i=-\infty}^{\lambda_{st}} Q^{\theta i}+\sum_{i=\lambda_{st}+1}^\infty \frac{Q^{\lambda_{st}+1}}{Q^{(1-\theta)i}} 
=\bigg(\frac{Q^\theta}{Q^\theta-1} +\frac{Q}{Q^{1-\theta}-1}\bigg) Q^{\theta \lambda_{st}}\stackrel{\eqref{Q large theta small}\wedge \eqref{eq:L beteen powers of Q}}{\lesssim} Q^\theta  L(s,t)^{\theta}.\tag*{\qedhere}
\end{align*}
\end{proof}

We can now complete the proof of Proposition~\ref{prop:Xcubed}:

\begin{proof}[Proof of Proposition~\ref{prop:Xcubed}] Denote $Q\eqdef (256K^2)^{3/\theta}$ and $\beta\eqdef 1/(32K^2)$. So, \eqref{Q large theta small} holds by design and we have $A_K/\theta\ge 32K^2(1+\log K)\log(eQ)$ for a suitable choice of $A_K\asymp K^2(\log K)^2\asymp_K 1$. Corollary~\ref{cor:all scales} shows that the assumptions of Lemma~\ref{fact:abstract}  hold for $L=d_\MM$; by its conclusion~\eqref{eq:F bookkeping guarantees}   the $\theta$-snowflake of $(\MM,d_\MM)$ embeds into $\X^3$ with distortion $Q^{2\theta}/\beta\asymp (256K^2)^{7}\asymp_K 1$. 
\end{proof}

Our proof of the second part of Theorem~\ref{thm:transition} will also use the following result:

\begin{prop}\label{prop:R2}
For every $\frac35\le \theta \le 1$ the $\theta$-snowflake of $\R$ embeds into $\R^2$ with distortion $O(1)$. 
\end{prop}

Proposition~\ref{prop:R2} is a well-known result arising from classical investigations on fractal curves; see also Remark~\ref{rem:theta distortion} below. As explained in~\cite{Gla58}, the von Koch snowflake construction~\cite{Koc06} yields Proposition~\ref{prop:R2} when $\theta=(\log 3)/\log 4$.~\footnote{Formally, \cite{Gla58}, as well as other literature on this topic, including e.g.~the works~\cite{Assouad,DH99,HM12} cited herein,  constructs an embedding of the $\theta$-snowflake of a closed interval, say,  $[-1,1]\subset \R$,  into $\R^2$. One can quickly pass as follows from such a statement to an embedding of the $\theta$-snowflake of the entire line $\R$ into $\R^2$. Considering  $f:[-1,1]\to \R^2$ that satisfies $|s-t|^\theta \le \|f(s)-f(t)\|_2\le D|s-t|^\theta$ for every $s,t\in [-1,1]$ and some $1\le D<\infty$, as well as $f(0)=0$, we may suppose that $f$ is defined continuously on  $\R$  by setting it to be constant on $\R\setminus [-1,1]$. The functions $\{(s\in \R)\mapsto n^\theta f(s/n)\}_{n=1}^\infty$ are equicontinuous and uniformly bounded on compact subsets of $\R$, so by Arzela--Ascoli they have a subsequential limit that is an embedding of the $\theta$-snowflake of $\R$ into $\R^2$ with distortion $D$ .}   An inspection of the proof in~\cite{Gla58} shows that it could be adapted mutatis mutandis to show   that for every $1/2<\theta\le 1$, the $\theta$-snowflake of $\R$ embeds into $\R^2$ with distortion $O(1/(2\theta-1))$, thus implying in particular Proposition~\ref{prop:R2}; a detailed justification of the  aforementioned distortion  bound for every $1/2<\theta\le 1$ has been carried out in~\cite{DH99,HM12}, and a more general treatment famously appears in~\cite{Assouad}, though it yields a weaker distortion upper bound of $O(1/(2\theta-1)^2)$, which also suffices for Proposition~\ref{prop:R2}.

\begin{remark}\label{rem:theta distortion}{\em By comparing Hausdorff dimensions one sees that for $0<\theta<1/2$ the $\theta$-snowflake of $\R$ does not admit a bi-Lipschitz embedding into $\R^2$. For $1/2<\theta\le 1$, let $\cc_2(\R^\theta)$ denote the smallest possible distortion of an embedding of the $\theta$-snowflake of $\R$ into $\R^2$. The best-known upper bound on $\cc_2(\R^\theta)$ when $1/2<\theta\le 1$ is the  aforementioned $\cc_2(\R^\theta)=O(1/(2\theta-1))$. At the endpoint case $\theta=1/2$, Hausdorff dimension is not an obstruction to the possibility that the $(1/2)$-snowflake of $\R$ admits a bi-Lipschitz embedding into $\R^2$. However,  it was proved in~\cite{BS61} by  more subtle  considerations that the   $(1/2)$-snowflake of $\R$ does not admit a bi-Lipschitz embedding into $\R^2$, whence $\lim_{\theta\to 0^+}  \cc_2(\R^{1/2+\e})=\infty$. Information on the rate at which this occurs does not seem to follow from~\cite{BS61}, but~\cite{Bro67} found a quantitative enhancement, from which one gets (by inspecting its proof; see specifically Lemma~3 there) that $\cc_2(\R^{1/2+\e})\gtrsim 1/\sqrt{\e\log(1/\e)}$ for every $0<\e\le 1/2$. The  best-known lower bound here is $\cc_2(\R^{1/2+\e})\gtrsim 1/\sqrt{\e}$, as seen  by combining the proof of~\cite[Theorem~4.1]{Vai81}  with~\cite[Theorem~2.1]{KP97}. In summary,  the best-available bounds here are: 
\begin{equation}\label{snowflake distortion}
\forall 0<\e\le \frac12,\qquad \frac{1}{\sqrt{\e}}\lesssim \cc_{\R^2}\big(\R^{\frac12+\e}\big)\lesssim \frac{1}{\e}.
\end{equation}
Determining the rate at which $\lim_{\e\to 0^+}\cc_2(\R^{\frac12+\e})=\infty$ remains  an interesting open question.}
\end{remark}

The following lemma strengthens~\cite[Remark~5.10]{MN04} in terms of the target dimension, though the embedding now has distortion $O(1)$ rather than being isometric; its special case $p\ge 2=q$ is  the   $\theta=2/p$ special case of the second part of Theorem~\ref{thm:transition}.

\begin{lemma}\label{lem:dimesnion version of MN} If $p\ge q\ge 1$ and $k\in \N$, then the $\frac{q}{p}$-snowflake of $\ell_q^k$ embeds with distortion $O(1)$ into $\ell_p^{O(pk/q)}$.
\end{lemma}

\begin{proof} By combining Proposition~\ref{prop:Xcubed} and Proposition~\ref{prop:R2} (for, respectively, the ranges $0<\theta=q/p\le 3/5$ and $3/5\le \theta=q/p\le 1$), we get an integer  $2\le m_{p,q}\lesssim p/q$ and a mapping $h_{p,q}:\R\to \R^{m_{p,q}}$ such that: 
\begin{equation}\label{eq:our hpq}
\forall s,t\in \R,\qquad \|h_{p,q}(s)-h_{p,q}(t)\|_p\asymp |s-t|^{\frac{q}{p}}. 
\end{equation}
Now simply consider the embedding  $f_{p,q}:\R^k\to (\R^{m_{p,q}})^k\cong \R^{km_{p,q}}$ that applies $h_{p,q}$ coordinate-wise:
\begin{equation}\label{eq:hpq tensor}
\forall x=(x_,\ldots,x_k)\in \R^k,\qquad f_{p,q}(x)\eqdef \big(h_{p,q}(x_1),\ldots,h_{p,q}(x_k)\big).
\end{equation}
Through the natural identification of $(\ell_p^{m_{p,q}})^{k}$ with $\ell_p^{km_{p,q}}$, any $x=(x_1,\ldots,x_k),y=(y_1,\ldots,y_k)\in \R^k$ satisfy: 
\begin{equation*}
\|f_{p,q}(x)-f_{p,q}(y)\|_{p}\stackrel{\eqref{eq:hpq tensor}}{=} \Big(\sum_{i=1}^k \|h_{p,q}(x_i)-h_{p,q}(y_i)\|_{p}^p\Big)^{\frac{1}{p}}\stackrel{\eqref{eq:our hpq}}{\asymp}\Big(\sum_{i=1}^k |x_i-y_i|^q\Big)^{\frac{1}{p}}=\|x-y\|_q^{\frac{q}{p}}. \tag*{\qedhere}
\end{equation*}
\end{proof}

The proof of the second part of Theorem~\ref{thm:transition}  composes the embeddings of Proposition~\ref{prop:Xcubed} and Proposition~\ref{prop:R2} with transformations that are  taken from (substantial) results in the literature on linear embeddings (covered in the survey in~\cite{JS01}); what we need for this is summarized in the following theorem:

\begin{theorem}\label{thm:FLM}  Suppose that   $0<q<\infty$ and $d\in \N$. There exists $n\in \N$ satisfying $n\lesssim d$ if $0<q\le 2$ and $n\lesssim_q d^{q/2}$ if $2<q<\infty$, and a linear mapping $T=T_{d,q}:\R^d\to \R^n$ such that:
\begin{equation}\label{eq:FLM}
\forall x\in \R^d,\qquad \|x\|_2\le \|Tx\|_q\lesssim e^{O\big(\frac{1}{q}\big)}\|x\|_2.
\end{equation}
\end{theorem}
When $1\le q <\infty$ the $e^{O(1/q)}$ term in~\eqref{eq:FLM} is $O(1)$, whence  Theorem~\ref{thm:FLM}   for $q$ in this range is due to~\cite{FLM}. If $0<q<1$,  then Theorem~\ref{thm:FLM} with the stated dependence on $q$ (which is needed below) follows from~\cite{Vyb08,FPRU10}. Even though~\cite{Vyb08,FPRU10} produce the desired mapping $T$, this is seen by inspecting the implicit constants that arise from the proofs therein rather than from a statement that is displayed explicitly in~\cite{Vyb08,FPRU10};  while verifying this is merely a mechanical exercise of unraveling notation and bookkeeping, it is worthwhile to include  next a derivation of it via a quick and standard extrapolation argument (examples of its uses for similar purposes can be found in~\cite{JS82,SZ01}) that reduces it to a theorem from~\cite{FLM,Kas77}.

By~\cite{FLM,Kas77} there is a universal constant $C\ge 1$, an integer $K\lesssim d$  and a linear subspace $E\subset\mathbb R^K$ with $\dim E = d$ such that:
\begin{equation}\label{eq:quote kashin} 
\forall y\in E,\qquad  \|y\|_1 \ge \frac{\sqrt{K}}{C}\|y\|_2.
\end{equation}
Since when we equip  $E$ with (any multiple of) the standard Euclidean norm on $\R^K$ we get a $d$-dimensional Hilbert space, there is a linear transformation $T:\R^d\to T$ that satisfies:
\begin{equation}\label{eq:identify with R^k}
\forall x\in \R^d,\qquad \|Tx\|_2=\bigg(\frac{C}{\sqrt{K}}\bigg)^{\frac{2}{q}-1}\|x\|_2.
\end{equation} 
Now, for every $x\in \R^d$ we have:
\begin{equation}\label{eq:use holder}
\bigg(\frac{C}{\sqrt{K}}\bigg)^{\frac{2}{q}-2}\|x\|_2\stackrel{\eqref{eq:identify with R^k}}{=}\frac{\sqrt{K}}{C}\|Tx\|_2\stackrel{\eqref{eq:quote kashin} }{\le}\|Tx\|_1\le \|Tx\|_q^{\frac{q}{2-q}}\|Tx\|_2^{\frac{2(1-q)}{2-q}}\stackrel{\eqref{eq:identify with R^k}}{=}\|Tx\|_q^{\frac{q}{2-q}} \bigg(\frac{C}{\sqrt{K}}\bigg)^{\frac{2}{q}-2}\|x\|_2^{\frac{2(1-q)}{2-q}} ,
\end{equation}
where the third step of~\eqref{eq:use holder} is an instantiation of H\"older's inequality with the (conjugate) exponents $2-q$ and $(2-q)/(1-q)$. By simplifying~\eqref{eq:use holder} we get the first inequality in~\eqref{eq:FLM}. The rest of~\eqref{eq:FLM} holds because:
$$
\|Tx\|_q\le K^{\frac{1}{q}-\frac12}\|Tx\|_2\stackrel{\eqref{eq:identify with R^k}}{=}C^{\frac{2}{q}-1}\|x\|_2,
$$
where we used  H\"older's inequality again, this time  with the conjugate exponents $2/q$ and $2/(2-q)$.

The following consequence of Proposition~\ref{prop:Xcubed}, Proposition~\ref{prop:R2} and Theorem~\ref{thm:FLM} is the  $p=2$ special case of the second part of Theorem~\ref{thm:transition}: 

\begin{lemma}\label{cor:euclidean arbitrary snowflake} For every $k\in \N$ and $0<\theta\le 1$ the $\theta$-snowflake of $\ell_2^k$ embeds with distortion $O(1)$ into $\ell_2^{O(k/\theta)}$.  
\end{lemma}

\begin{proof} By combining Proposition~\ref{prop:Xcubed} and Proposition~\ref{prop:R2} (for  $0<\theta\le 3/5$ and $3/5\le \theta\le 1$, respectively), we see that there is an integer  $2\le d_\theta\lesssim 1/\theta$ and a mapping $h_\theta:\R\to \R^{d_\theta}$ such that: 
\begin{equation}\label{eq:2/p h}
\forall s,t\in \R,\qquad \|h_{\theta}(s)-h_\theta (y)\|_2\asymp |s-t|^{\theta}. 
\end{equation}
By  Theorem~\ref{thm:FLM} for $q=2\theta\le 2$ there is an integer $1\le n\lesssim k$ and a  linear mapping $T_\theta: \R^k\to \R^n$ satisfying:
\begin{equation}\label{eq:T theta}
\forall x\in \R^k,\qquad \|T_\theta x\|_{2\theta}^\theta\asymp \|x\|_2^\theta. 
\end{equation}
Note that the dependence on $q$ in~\eqref{eq:FLM} was used crucially here. 

Write $T_\theta x=((T_\theta x)_1,\ldots,(T_\theta x)_n)$ for  $x\in \R^k$. Setting  $m_\theta= nd_\theta\lesssim k/\theta$ define  $f_\theta:\R^k\to (\R^{d_\theta})^n\cong \R^{m_\theta}$ by:
\begin{equation}\label{eq:def f theta}
\forall x=(x_1,\ldots,x_k)\in \R^k,\qquad f_\theta(x)\eqdef \Big(h_\theta\big((T_\theta x)_1\big),\dots,h_\theta \big((T_\theta x)_n\big)\Big).
\end{equation}
Then, $f_\theta$ demonstrates the desired embedding conclusion as every $x,y\in \R^k$ satisfy:
\begin{align*}
\|f_\theta(x)-f_\theta(y)\|_2=\bigg(\sum_{i=1}^n \Big\|&h_\theta\big((T_\theta x)_i\big)-h_\theta\big((T_\theta y)_i\big)\Big\|_2^2\bigg)^{\frac12}\\&\stackrel{\eqref{eq:2/p h}}{\asymp} \bigg(\sum_{i=1}^n \big|(T_\theta x)_i-(T_\theta y)_i\big|^{2\theta}\bigg)^{\frac12}=\|T_\theta x-T_\theta y\|_{2\theta}^\theta\stackrel{\eqref{eq:T theta}}{\asymp} \|x-y\|_2^\theta. \tag*{\qedhere}
\end{align*}
\end{proof}

With the above statements at hand, we can now complete the proof  of Theorem~\ref{thm:transition}:

\begin{proof}[Proof of  the second part of Theorem~\ref{thm:transition}] Fix  $1\le p<\infty$ and $0<\theta\le 1$, as well as $k\in \N$. Our goal is to show that the $\theta$-snowflake of $\ell_2^k$ embeds into $\ell_p^m$ for some integer $m$ that satisfies $m\lesssim_p k/\theta$ if $p\le 2/\theta$, and in the range  $p\ge 2/\theta$, in which case $\theta\asymp_p 1$,  we have  $m\lesssim_p (k/\theta)^{p\theta/2}\asymp k^{p\theta/2}$.

Suppose first that $1\le p\le 2$. Use Lemma~\ref{cor:euclidean arbitrary snowflake}  to get an integer $K\asymp k/\theta$ such that the $\theta$-snowflake of $\ell_2^k$ embeds into $\ell_2^K$ with distortion $O(1)$. By Theorem~\ref{thm:FLM} we know that $\ell_2^K$, whence also the $\theta$-snowflake of $\ell_2^k$, embeds with distortion $O(1)$ into $\ell_p^m$ for some integer $m\asymp K\asymp k/\theta$, as required. 

If $p\ge 2/\theta$ (whence $\theta\asymp_p 1$), then  Theorem~\ref{thm:FLM} for $q=p\theta\ge 2$ shows that $\ell_2^k$ embeds with distortion $O(1)$  into $\ell_{p\theta }^d$ for some integer $d\asymp_{p} k^{p \theta /2}$. By Lemma~\ref{lem:dimesnion version of MN}  the $\theta$-snowflake of $\ell_{p\theta}^d$, whence also the $\theta$-snowflake of $\ell_2^k$, embeds with distortion $O(1)$ into $\ell_p^{m}$ for some integer $m\lesssim d/\theta\asymp_p d\asymp_p k^{p\theta/2}$, as required.

It remains to treat the case  $2\le p\le 2/\theta$.  Because $p\theta/2\le 1$, we may apply Lemma~\ref{cor:euclidean arbitrary snowflake} with $\theta$ replaced by $p\theta/2$ to get an integer $\ell\asymp k/(p\theta)$ and  $f: \R^k\to \R^\ell$ that satisfies:  
\begin{equation}\label{eq:tau snowflake}
\forall x,y\in \R^k,\qquad \|f(x)-f(y)\|_2\asymp \|x-y\|_2^{\frac{p\theta}{2}}.
\end{equation}
As $p\ge 2$,  apply Lemma~\ref{lem:dimesnion version of MN} with $q=2$ to get an integer $1\le m\asymp p\ell\asymp k/\theta$  and $g:\R^\ell\to \R^m$ satisfying:
\begin{equation}\label{eq:2 over p snowflake}
\forall x,y\in \R^\ell,\qquad \|g(x)-g(y)\|_p\asymp \|x-y\|_2^{\frac{2}{p}}.
\end{equation}
The composition $g\circ f:\R^k\to \R^m$ now satisfies the desired conclusion as follows: 
\begin{equation*}
\forall x,y\in \R^k,\qquad \|g\circ f(x)-g\circ f(y)\|_p\stackrel{\eqref{eq:2 over p snowflake}}{\asymp} \|f(x)-f(y)\|_2^{\frac{2}{p}} \stackrel{\eqref{eq:tau snowflake}}{\asymp} \|x-y\|_2^\theta. \tag*{\qedhere}
\end{equation*}
\end{proof}

\section{On the sharp Assouad problem}\label{sec:sharp assouad}

Fix $K\in \N$ and  $0<\theta<1$. By Assouad's  embedding theorem~\cite{Assouad}, the $\theta$-snowflake of any $K$-doubling metric space admits a bi-Lipschitz embedding into some Euclidean space, where the distortion and dimension depend only on $K$ and $\theta$. When $1/2\le \theta<1$, by~\cite{NN} (see also~\cite{DS13} for a different proof), the aforementioned target dimension can be taken to be independent of $\theta$, but for small $\theta$ comparison of Hausdorff dimensions shows that the smallest target dimension that one could hope for in Assouad's embedding theorem is at least of order  $(\log K)/\theta$. The order of magnitude of the smallest possible distortion that one could achieve in Assouad's embedding theorem if one requires the target dimension to be of that smallest possible order of magnitude is the content of the following conjecture: 

\begin{conjecture}\label{conj:sharp assouad} For every integer $K\ge 3$ and every $0<\theta<1$, the $\theta$-snowflake of any $K$-doubling metric space embeds with distortion $D$ into $\ell_2^n$, where $D\ge 1$ and $n\in \N$ satisfy:  
\begin{equation}\label{eq:sharp sharp assouad}
n\asymp \frac{\log K}{\theta}\qquad\mathrm{and}\qquad D\asymp \frac{(\log K)^\theta}{\sqrt{1-\theta}}\asymp \left\{\begin{array}{ll}1&\mathrm{if\ } 0<\theta\le \frac{1}{\log \log K},\\
\frac{(\log K)^\theta}{\sqrt{1-\theta}} &\mathrm{if\ } \frac{1}{\log \log K}\le \theta \le 1-\frac{1}{\log K},\\
\frac{\log K}{\sqrt{1-\theta}}&\mathrm{if\ } 1-\frac{1}{\log K}\le \theta<1.\end{array}\right.
\end{equation}
Furthermore, this statement cannot be improved up to the values of the implicit universal constants in~\eqref{eq:sharp sharp assouad}. 
\end{conjecture}

Unlike  the sharp Lang--Plaut problem, in~\eqref{eq:sharp sharp assouad} the distortion must depend on $K$ and $\theta$ because Conjecture~\eqref{conj:sharp assouad} treats all $K$-doubling metric spaces,  including those that do not embed (with any distortion) into Hilbert space (such spaces exist even when $K=O(1)$ by~\cite{Pan89,Sem99}). Nevertheless,  for $0<\theta\lesssim 1/\log\log K$,  Conjecture~\eqref{conj:sharp assouad} does imply that the $\theta$-snowflake   of any $K$-doubling metric space embeds with distortion $O(1)$ into $\ell_2^n$ of the optimal dimension  $n\lesssim (\log K)/\theta$, and furthermore, this holds only for  $\theta$ in this range.  That conclusion is not conjectural, thanks to the following remark, which describes the best bounds that we currently have towards Conjecture~\eqref{conj:sharp assouad}:

\begin{remark}\label{rem:known sharp assouad} {\em For every integer  $K\ge 3$ and every  $0<\theta,\d< 1$, any $K$-doubling metric space embeds with distortion $D$ into $\ell_2^n$, where $D\ge 1$ and $n\in \N$ satisfy:  
\begin{equation}\label{eq:known upper sharp assouad}
n\asymp \frac{\log K}{\d\theta}\qquad\mathrm{and}\qquad  D\lesssim \left\{ \begin{array}{ll} 1 &\mathrm{if\ }\ 0<\theta\le \frac{1}{\log\log K},\\ \big((\log K)\log\log K\big)^{(1+\d)\theta} &\mathrm{if\ }\  \frac{1}{\log\log K}\le \theta\le 1-\frac{1}{\log \log K},\\
\left(\frac{\log K}{1-\theta}\right)^{1+\d}&\mathrm{if\ }\  1-\frac{1}{\log \log K}\le \theta<1.\end{array}\right.
\end{equation}
To see why this holds, fix  $\max\{\theta,1/2\}<\tau<1$. By~\cite{NN} the $\tau$-snowflake of any $K$-doubling metric space  embeds with distortion  $D_0 \asymp ((\log K)/(1-\tau))^{1+\d}$ into $\ell_2^m$, where  $m\lesssim (\log K)/\d$.  As $0<\theta/\tau\le 1$, we may use  Lemma~\ref{cor:euclidean arbitrary snowflake} to embed the $(\theta/\tau)$-snowflake of $\ell_2^m$ into $\ell_2^n$, where $n\in \N$ satisfies $n\lesssim m\tau/\theta\asymp (\log K)/(\d\theta)$.  By composing these two embeddings, we see that a $K$-doubling metric space  embeds into $\ell_2^n$ with distortion  $D=D_0^{\theta/\tau}\asymp ((\log K)/(1-\tau))^{(1+\d)\theta/\tau}$. By choosing the $\tau$ that minimizes the latter distortion, we get~\eqref{eq:known upper sharp assouad}.

Conversely, in the context of Conjecture~\ref{conj:sharp assouad} we must have $D\gtrsim (\log K)^\theta$, as seen by considering expander graphs and reasoning as in~\cite{LLR95}; this demonstrates  the aforementioned assertion that $D=O(1)$ is impossible unless $0<\theta\lesssim 1/\log\log K$. By~\cite{LMN}, in the context of Conjecture~\ref{conj:sharp assouad}  we must also have $D\gtrsim 1/\sqrt{1-\theta}$. In fact, for arbitrarily large  $K\in \N$ there is a metric space $\MM_K$ such that for every $0<\theta<1$, if the $\theta$-snowflake of $\MM_K$ embeds with distortion $D\ge 1$ into an infinite dimensional Hilbert space, then: 
\begin{equation}\label{eq:with log log}
D\gtrsim \frac{(\log K)^\theta}{\sqrt{\log \log K}}\cdot \frac{1}{\sqrt{1-\theta}}.
\end{equation}
Indeed, the example that is constructed in~\cite{JLM09} has this property; while~\cite{JLM09} treats only bi-Lipschitz mappings, a verbatim repetition of the analysis in~\cite{JLM09} yields the following bound, which is stronger than~\eqref{eq:with log log}:
\begin{align}\label{eq:cases log log}
D&\gtrsim \left\{\begin{array}{ll}1&\mathrm{if\ } 0<\theta\le \frac{1}{\log \log K},\\
(\log K)^\theta &\mathrm{if\ } \frac{1}{\log \log K}\le \theta\le 1-\frac{1}{\log \log K},\\
\frac{\log K}{\sqrt{(1-\theta)\log \log K}}&\mathrm{if\ } 1-\frac{1}{\log \log K}\le \theta <1.\end{array}\right.
\end{align}

Conceivably \eqref{eq:with log log}  could be improved to $D\gtrsim (\log K)^\theta/\sqrt{1-\theta}$, but Conjecture~\ref{conj:sharp assouad} posits less, namely, it asks for such a distortion lower bound (for any $K$-doubling metric space, not necessarily the example of~\cite{JLM09}) under the more stringent requirement to embed into $\ell_2^n$ for $n\lesssim (\log K)/\theta$, while the lower bound~\eqref{eq:cases log log} holds for embedding into $\ell_2$.   
}
\end{remark}

Remark~\ref{rem:known sharp assouad} contains  the best evidence that we  have for Conjecture~\eqref{conj:sharp assouad}. It comes quite close to Conjecture~\eqref{conj:sharp assouad} when $0<\theta<1$ is bounded away from $1$,  but it is much less satisfactory when $\theta\to 1^-$, which is especially important. Specifically, Conjecture~\eqref{conj:sharp assouad} predicts that the distortion remains $O_K(1/\sqrt{1-\theta})$ as $\theta\to 1^-$, which  has been a longstanding open problem even if we relax the requirement that the target dimension is the optimal $O(\log K)$ by allowing it to be any function of $K$ whatsoever (see e.g.~\cite{Tao21,Ryo22,Ryo23}).

\bibliographystyle{abbrv}
\bibliography{noJL}

\end{document}